\documentclass[a4paper,11pt]{amsart}

\usepackage[english]{babel}
\usepackage{amsmath,amssymb,amsthm,amscd}
\usepackage[alphabetic]{amsrefs}
\usepackage{enumitem}
\usepackage[unicode,pdfencoding=auto,psdextra]{hyperref}
\usepackage{xcolor}
\usepackage{mathrsfs} 

\usepackage{hyperref}
\usepackage{orcidlink}

\usepackage{kantlipsum} 

\calclayout

\usepackage{etoolbox}
\apptocmd{\sloppy}{\hbadness 10000\relax}{}{}


\newtheorem{theorem}{Theorem}[section]

\newtheorem{lemma}[theorem]{Lemma}
\newtheorem{proposition}[theorem]{Proposition}
\theoremstyle{definition}
\newtheorem{definition}[theorem]{Definition}

\theoremstyle{remark}
\newtheorem{remark}[theorem]{Remark}

\newcommand{\CC}{\mathbb{C}} 
\newcommand{\C}{\mathbb{C}}
\newcommand{\RR}{\mathbb{R}} 

\newcommand{\T}{\mathrm{T}}

\newcommand{\NN}{\mathbb{N}}
\newcommand{\camo}{\mathcal{C}_n}

\newcommand{\holo}{\mathcal{O}}
\newcommand{\matnn}{\mathrm{M}_n( \mathbb{C})}
\newcommand{\matnR}{\mathrm{M}_n( \mathbb{R})}
\newcommand{\mabove}{\mathcal{M}}
\newcommand{\maboveR}{\mathcal{M}_{\mathbb{R}}}
\newcommand{\cabove}{\hat{\mathcal{C}}_n}
\newcommand{\caboveR}{\hat{\mathcal{C}}_n^\mathbb{R}}
\newcommand{\caboveF}{\tilde{\mathcal{C}}_n}
\newcommand{\caboveFR}{\tilde{\mathcal{C}}_n^\mathbb{R}}
\newcommand{\caboveB}{\tilde{\mathcal{C}}'_n}
\newcommand{\caboveBr}{\tilde{\mathcal{C}}_n^{r}}
\newcommand{\ssubset}{\subset\joinrel\subset}
\newcommand{\dif}{\mathrm{d}}
\newcommand{\taut}{\tau\mathrm{Aut}}
\DeclareMathOperator{\tr}{tr}
\DeclareMathOperator{\rank}{rank}
\DeclareMathOperator{\id}{id}
\DeclareMathOperator{\re}{Re}
\DeclareMathOperator{\im}{Im}
\DeclareMathOperator{\gl}{GL}
\DeclareMathOperator{\gln}{G}
\DeclareMathOperator{\un}{U}
\DeclareMathOperator{\aut}{Aut}

\DeclareMathOperator{\diff}{Diff}

\DeclareMathOperator{\spec}{Spec}

\newcommand\restr[2]{{
  \left.\kern-\nulldelimiterspace 
  #1 
  \littletaller 
  \right|_{#2} 
  }}

\newcommand{\littletaller}{\mathchoice{\vphantom{\big|}}{}{}{}}

\title[Holomorphic Approximation of Symplectic Diffeomorphisms]{Holomorphic Approximation of Symplectic Diffeomorphisms for Calogero--Moser Spaces}

\author{Gaofeng Huang}
\address{Mathematical Institute \\ University of Bern \\  
Bern, Switzerland}
\email{gaofeng.huang@unibe.ch}

\thanks{The author was supported by Schweizerisches Nationalfonds [200021-207335].}

\begin{document}


\begin{abstract}
{The real Calogero--Moser space $\mathcal{C}_n^\mathbb{R}$ is a noncompact, totally real submanifold of the complex Calogero--Moser space $\mathcal{C}_n$. 
We prove that every symplectic diffeomorphism of $\mathcal{C}_n^\mathbb{R}$ smoothly isotopic to the identity can be approximated in the fine Whitney topology -- the strongest in this context -- by holomorphic symplectic automorphisms of $\mathcal{C}_n$ that preserve $\mathcal{C}_n^\mathbb{R}$. A key ingredient in our proof is a refined version of the symplectic density property of $\mathcal{C}_n$.}
\end{abstract}

\maketitle

\tableofcontents

\section{Introduction}

We study approximation for symplectic diffeomorphisms of the real Calogero--Moser space $\camo^\RR$ onto itself by holomorphic symplectic automorphisms on the complex Calogero--Moser space $\camo$, $n \in \NN$.

The complex Calogero--Moser space is a complex affine manifold, endowed with a holomorphic symplectic form, describing the phase space of $n$ indistinguishable particles on a plane interacting pairwise via an inverse square potential. 

The real Calogero--Moser space $\camo^\RR$ was constructed by Kazhdan, Kostant and Sternberg \cite{MR478225} as a real symplectic reduction (cf.\ Definition \ref{definition: realCM}). Two decades later, the complex Calogero--Moser space $\camo$ was introduced by Wilson \cite{MR1626461} in the complex setting as follows. 

Let $\mabove$ be the direct sum $\matnn \oplus \matnn \oplus \CC^n \oplus (\CC^n)^\ast$, where $\matnn$ is the $\CC$-vector space of square matrices of size $n$. The space $\mabove$ can be endowed with the holomorphic symplectic form $\omega = \tr ( d X \wedge dY + d v \wedge d w )$. Moreover, the general linear group $\gl_n(\CC)$ acts on $\mabove$ as
\begin{align*}
	g \cdot (X, Y, v, w ) = ( g X g^{-1}, g Y g^{-1}, g v, w g^{-1} ), \, \, g \in \gl_n(\CC)
\end{align*}
which preserves the symplectic form. Thus this action induces a complex moment map 
\begin{align*}
	\mu \colon \mabove \to \matnn,  \,
	(X, Y, v, w) \mapsto [X, Y] + v w
\end{align*}
which is equivariant with respect to the above action on $\mabove$ and the coadjoint action on $\matnn \cong \mathfrak{gl}_n^\ast$.
Take the coadjoint invariant point $ i I_n \in \matnn$. By Wilson \cite{MR1626461}, the group action is free on the preimage $\mu^{-1}(i I_n)$.

\begin{definition}
	The \emph{Calogero--Moser space} $\camo$ is the complex symplectic reduction $\mu^{-1} (  i I_n) / \gl_n(\CC)$.
\end{definition}	

Let $\tau \colon \camo \to \camo$ be the antiholomorphic involution given by 
\[
    ( X, Y, v, w) \mapsto ( X^\ast, Y^\ast, i w^\ast, i v^\ast).
\]
Here, the star stands for complex conjugate transpose. Then the fixed-point set of $\tau$ turns out to be real symplectomorphic to the real Calogero--Moser space $\camo^\RR$. In other words, the complex space $\camo$ is a symplectic complexification of $\camo^\RR$, see Definition \ref{definition: realform}. 

\begin{theorem} \label{theorem: realform}
    The complex Calogero--Moser space $(\camo, \omega, \tau)$ is a symplectic complexification of the real Calogero--Moser space $(\camo^\RR, \omega_\RR)$. 
\end{theorem}

\smallskip

Our approximation is what complex analysts call the Carleman approximation. 
It is an approximation of real objects by holomorphic ones in the fine (Whitney) topology.
Torsten Carleman was the first to obtain such a result. He proved that smooth functions on the real line can be approximated by holomorphic functions on the complex plane in the fine topology \cite{carleman1927}. For the following definition we fix a norm $\| \cdot \|_{C^k}$ on the jet-space $\mathcal{J}^k(M)$, cf. Section \ref{section: prelim}.

\begin{definition}
Let $(M_\RR, \omega_\RR)$ be a smooth symplectic manifold and $(M, \omega, \tau)$ a symplectic complexification of $(M_\RR, \omega_\RR)$. For a natural number $k$, we say that $(M, \omega, \tau)$ admits \emph{Hamiltonian $C^k$-Carleman approximation}, if 
for any Hamiltonian diffeomorphism $\varphi$ of $M_\RR$ onto itself (cf.\ Section \ref{section: prelim}) and any positive continuous function $\epsilon$ on $M_\RR$, there exists a holomorphic symplectic automorphism $\Phi$ of $M$ such that $\Phi(M_\RR) = M_\RR$ and the estimate $\| \Phi - \varphi \|_{C^k(p)} < \epsilon (p)$ holds for any $p \in M_\RR$.
\end{definition}

The main result of this paper is: 
\begin{theorem}
    The Calogero--Moser space $(\camo, \omega, \tau)$ admits Hamiltonian $C^k$-Carleman approximation for all $k \in \NN_0$. 
\end{theorem}

For such an approximation to hold, the complex symplectic manifold $M$ must admit a large group of holomorphic symplectic automorphisms. A key ingredient in establishing Carleman approximation for automorphisms is the Anders{\'e}n--Lempert theory, which concerns the approximation of biholomorphic mappings between Runge domains by holomorphic symplectic automorphisms. 
A precise way in the Anders{\'e}n-Lempert theory  to express the abundance of holomorphic automorphisms is Varolin's density property \cite{MR1829353}. The following is its symplectic counterpart.  

\begin{definition} \label{definition: SDP}
A complex symplectic manifold $(M, \omega)$ is said to have \emph{the symplectic density property} if the Lie algebra generated by $\CC$-complete holomorphic symplectic vector fields is dense in the Lie algebra of all holomorphic symplectic vector fields in the compact-open topology. 
\end{definition}
 
Known nontrivial examples of Stein symplectic manifolds satisfying the symplectic density property include even-dimensional complex Euclidean space \cite{MR1408866} 
and the Calogero--Moser space \cite{CaloSymplo}. In both cases, the vanishing of the first de Rham cohomology reduces the property to \emph{the Hamiltonian density property}, where only Hamiltonian vector fields are considered. The Hamiltonian density property was also established for closed coadjoint orbits of complex Lie groups in  \cite{MR4423269}. 
This property allows approximation of real Hamiltonian vector fields on compact subsets by  sums of holomorphic Hamiltonian vector fields, whose real parts are tangent to a given totally real submanifold; see \cite{MR4423269}*{Lemma 3.1}. 

\smallskip

Carleman approximation of functions and maps has a long tradition in complex analysis, for a comprehensive overview, see \cite{MR4264040}, and for a recent result, see \cite{MR4459548}. 

There have been only two results on  Carleman approximation of smooth diffeomorphisms on noncompact totally real submanifolds. 
Recall that a $C^1$ real submanifold $N$ of a complex manifold $M$ is called totally real, if at each point $p$ of $N$ the tangent space $T_p N$ contains no complex line. 
The first result is by Kutzschebauch and Wold \cite{MR3794890} who proved Carleman approximation of diffeomorphisms from $\RR^s$ onto itself by holomorphic automorphisms of $\CC^n$ in the fine topology when $s < n$. Their proof relies on Anders{\'e}n--Lempert theory and Carleman approximation by entire functions. This result generalized earlier works on approximation by holomorphic automorphisms on compact totally real submanifolds, including Forstneri\v{c}--Rosay \cite{MR1213106}, Forstneri\v{c} \cite{MR1314745} and Forstneri\v{c}--L\o w--\O vrelid \cite{MR1852309}. In particular, \cite{MR1852309} studied the approximation of certain diffeomorphisms between compact totally real submanifolds by holomorphic $\omega$-automorphisms of $\CC^{2n}$, where $\omega$ denotes either the standard holomorphic volume form or the holomorphic symplectic form. 

The second result is by Deng and Wold \cite{MR4423269} who established Hamiltonian Carleman approximation for closed coadjoint orbits of complex Lie groups. Their approach not only yields holomorphic symplectic automorphisms preserving the real part but also addresses the case $\RR^{2n}$ embedded in $\CC^{2n}$ with the standard symplectic form $\omega = \sum_{j = 1}^n dz_j \wedge d w_j$. The proofs of all these results rely heavily on the density property of $\CC^n$ or on the 
symplectic density property of $\CC^{2n}$.
Even the result in \cite{MR4423269} on coadjoint orbits 
depends on it, as it is obtained by restricting the objects from the coadjoint representation, which symplectically is isomorphic to $\CC^{2n}$, to the coadjoint orbits. 

Our result is the first one where the ambient complex space is much more complicated than affine space. A crucial component of our proof is the fact that the
symplectic density property (see Definition \ref{definition: SDP}) of the Calogero--Moser space is proven using complete holomorphic vector fields, whose real-time flows preserve the real Calogero--Moser space; see Definition \ref{def-tangential-dp}.

\bigskip

The article is structured as follows: Section \ref{section: prelim} introduces the setup. In Section \ref{section: realform}, we prove  Theorem \ref{theorem: realform}, which claims that the real Calogero--Moser space $\camo^\RR$ has the complex one as a symplectic complexification. 
Section \ref{section: tangencyfunction} applies the symplectic density property of $\camo$ to express real algebraic symplectic vector fields as Lie combinations of complete symplectic vector fields (Proposition \ref{proposition: realADP}). This allows for local approximation of real vector fields on $\camo^\RR$ by real Lie combinations of complete vector fields whose real-time flows preserve $\camo^\RR$ (Proposition \ref{proposition: approxRbyC}). Section \ref{section: localVFnAUTO} extends this local approximation to automorphisms on Saturn-like subsets of $\camo$. This step is crucial in the push-out method, which constructs a holomorphic automorphism that is close to the identity inside a complex ball (Theorem \ref{theorem: LocalCarleman}). 
Section \ref{section: globalAUTO} combines local approximation with the push-out method to obtain global approximation of symplectic diffeomorphisms by holomorphic automorphisms (Theorem \ref{theorem: HamCarleman}). The Appendix examines an alternative totally real submanifold for which the complex Calogero--Moser space serves as a complexification.

\section{Preliminaries} \label{section: prelim}

\subsection{Hamiltonian diffeomorphism}\label{hamdiffeo}
The objects to be approximated are certain symplectic diffeomorphisms which are isotopic to the identity. 

Let $(M, \Omega)$ be a smooth symplectic manifold without boundary. A \emph{symplectic isotopy} is a jointly smooth map $\phi \colon [0, 1] \times M \to M$ such that $\phi_t$ is symplectic for every $t$ in $[0,1]$ and $\phi_0 = \id$. The isotopy $\phi_t$ is generated by a smooth family of vector fields $V_t$ with $d \phi_t / dt= V_t \circ \phi_t, \, \phi_0 = \id$. As $\phi_t$ is symplectic, the vector fields $V_t$ are symplectic by Cartan's formula
\begin{align*}
		0 =  \frac{d}{dt}  \phi_t^\ast \Omega  = \phi_{t}^\ast L_{V_{t}} \Omega = \phi_t^\ast d( \iota (V_t) \Omega )
\end{align*} 
where $L_{V_t}$ is the Lie derivative with respect to $V_t$. A \emph{Hamiltonian isotopy} is a symplectic isotopy such that the closed one-form $\iota(V_t) \Omega$ is exact, namely $V_t$ is Hamiltonian, for all $t \in [0,1]$. If the manifold $M$ is simply connected, then every symplectic isotopy is Hamiltonian (the smooth $t$-dependence of the Hamiltonian functions can be achieved by fixing a reference point $p_0$ in $M$ and choosing $H_t(p_0) = 0$ for all $t$ in $[0,1]$).  
Moreover, we call a symplectomorphism which is the endpoint of a Hamiltonian isotopy a \emph{Hamiltonian diffeomorphism}. Hamiltonian diffeomorphisms form a normal subgroup of the group of symplectic diffeomorphisms, see \cite[\S 3.1]{MR3674984}. 

\subsection{Pointwise seminorm}
Next, we explain the notation $\| \cdot \|_{C^k(\cdot)}$, $k \ge 1$,  following Manne--Wold--\O vrelid \cite[\S 2.2]{MR2854106} and refer to \cite[Ch. II]{MR0341518} for details. Let $V \subset M$ be a subset and $p \in V$. Consider the equivalence relation on germs of $C^k$-smooth complex-valued functions at $p$: $f_p \sim g_p$ if and only if $f-g$ vanishes to $k$-th order at $p$. Denote by $\mathcal{J}_p^k$ the set of equivalence classes, which forms a finite dimensional vector space. The union 
\[  \mathcal{J}^k(M, V) = \bigcup_{p \in V} \mathcal{J}_p^k
\]
can be endowed with the structure of a complex vector bundle over $V$. Each $C^k$-smooth function $f$ on $M$ induces a continuous section $\mathcal{J}^k(f)$ of $\mathcal{J}^k(M, V)$ by $\mathcal{J}^k(f)(p) = [f_p]$. Choose fiberwise a norm $\lvert \cdot \rvert$ on $\mathcal{J}^k(M, V)$ which varies continuously with respect to $p$. The pointwise seminorm $\| \cdot \|_{C^k(\cdot)}$ for $f$ is simply the norm of the induced $k$-jet
\[
    \| f \|_{C^k(p)} := \lvert \mathcal{J}^k(f)(p) \rvert
\] 
For a compact $K \subset M$, we set 
\[
    \| f \|_{C^k(K)} := \sup_{p \in K} \| f \|_{C^k(p)}
\]
For a smooth mapping $\Phi \colon M \to M$ the pointwise seminorm with respect to a local smooth chart $\alpha = ( \alpha_1, \alpha_2, \dots, \alpha_m) \colon U \to \alpha(U) \subset \RR^{m}$ is 
\[
    \| \Phi \|_{C^k(p)} := \sum_{j = 1}^m  \| \alpha_j \circ \Phi \|_{C^k(p)}
\]
Similarly we have $\| \Phi \|_{C^k(K)}$ by taking the supremum over a compact $K$. A different local chart $\beta \colon U' \to \beta(U')$ yields another pointwise seminorm $\| \cdot \|_{C^k(\cdot)}'$ satisfying 
\[     C(p)^{-1} \| \Phi \|_{C^k(p)} \le \| \Phi \|_{C^k(p)}'  \le C(p) \| \Phi \|_{C^k(p)}
\]
where $C$ is a positive continuous function on $U \cap U'$, therefore preserving the approximation on compact subsets (up to rescaling). 

\subsection{Holomorphic vector fields} Let $M$ be a complex manifold. A real vector field is a section of the real tangent bundle $\T M$, while a complex vector field is a section of the complexified tangent bundle $\T M \otimes_{\RR} \CC$. Through the almost complex structure $J$ on $\T M$, the complex tangent bundle decomposes into the holomorphic and antiholomorphic subbundles
\[ 
    \T M \otimes_\RR \CC = \T^{1,0}M \oplus \T^{0,1}M
\]
We can identify the real tangent bundle $TM$ and the holomorphic tangent bundle $\T^{1,0}M$ via the $\RR$-linear isomorphism (see e.g. \cite[\S 1.6]{MR3700709})
\[
    \alpha \colon \T M \hookrightarrow \T M \otimes_\RR \CC \to \T^{1,0}M, \, V \mapsto \frac{1}{2}( V - i JV)
\]
which has the inverse $\alpha^{-1}(W) = 2 \re (W)$ for $W \in \T^{1,0}M$. We have the following commutative diagram
\[
\begin{CD}
	\T M @>J>>	 \T M \\
	@V\alpha VV 			 @VV \alpha V\\
	\T^{1,0} M	@>i>>	\T^{1,0} M
\end{CD}
\]
\smallskip
A real vector field $V$ on $M$ is called \emph{holomorphic} if $\alpha(V)$ is a holomorphic section of $\T^{1,0}M$.

\begin{definition} \label{def: real-tangential}
    Let $\tau \colon M \to M$ be an antiholomorphic involution and $N$ be the fixed point set of $\tau$, which is a totally real submanifold. 
    A holomorphic function $f \in \holo(M)$ is called \emph{$\tau$-compatible}  if $\tau^\ast f= \overline{f}$.  
    A holomorphic vector field $V$ on $M$ is \emph{$\tau$-compatible} 
    if $\tau_\ast V = \overline{V}$.
\end{definition}

A $\tau$-compatible function is real-valued on the fixed point set $N$. In the Hamiltonian setting $\tau$-compatible functions and $\tau$-compatible vector fields are closely related, cf.\ Lemma \ref{lem:Rholo-vs-Rtang}.

\begin{lemma} \label{lem:realFLOWtang}
    Let $V$ be a holomorphic vector field on $M$ with $(N, \tau)$ as above and $N \neq \emptyset$. The vector field $V$ is $\tau$-compatible if and only if $\alpha^{-1}(V_p) = 2 \re (V_p) \in \T_p N$ for all $p \in N$. In particular, the $\RR$-flow of a $\tau$-compatible vector field preserves the submanifold $N$. 
\end{lemma}
\begin{proof}
Let $x$ be a point in $N$. Consider the flow equation of $V$
\begin{align*}
    V(x) = \restr{\frac{\partial}{\partial t} \Phi(t,x) }{t=0}
\end{align*}
Under $\tau_\ast$ it becomes
\begin{align*}
    \overline{V}(x) = \tau_\ast V(x) = \restr{\frac{\partial}{\partial t} \tau \circ \Phi(t,x) }{t=0}
\end{align*}
Adding these two equations and writing the complex time as $t = u + iv$ yields
\begin{align*}
    2 \re V(x) = \restr{\frac{\partial}{\partial t} (\id + \tau) \circ \Phi(t,x) }{t=0} = \restr{\frac{1}{2}\left(\frac{\partial}{\partial u} - i \frac{\partial}{\partial v}\right) (\id + \tau) \circ \Phi(t,x) }{t=0}
\end{align*}
When only flow in real time is considered
\begin{align*}
    2 \re V(x) = \restr{\frac{1}{2} \frac{\partial}{\partial u} (\id + \tau) \circ \Phi(u,x) }{u=0} 
\end{align*}
On the other hand, $2 \re V(x) = \alpha^{-1}(V)$ is the corresponding smooth vector field for $V$. The flow equation for this smooth vector field in real time is
\begin{align*}
    2 \re V(x) = \restr{\frac{\partial}{\partial u} \phi(u,x) }{u=0} 
\end{align*}
where $\phi(u, x) = \Phi(j(u), x)$ with the embedding $j \colon \RR \to \CC, t \mapsto (t, 0)$. Therefore $\tau \circ \Phi(t, x) = \Phi(t, x)$ for $t \in \RR$ and $x \in N$ as well as $\re V(x) \in \T_x N$. 

Backtracking the argument above and using $\RR$ totally real in $\CC$ shows $\tau_\ast V = \overline{V}$ on $N$. Since $N = \mathrm{Fix}(\tau) \neq \emptyset$ has real dimension $n = \dim_\C M$ by \cite[Proposition 1.1]{MR0589903}, it is totally real of maximal dimension in $M$. Then the holomorphicity of $V - \overline{\tau_\ast V}$ implies $\tau_\ast V = \overline{V}$ on $M$. 
\end{proof}

\subsection{Complexification}
We need the notion of complexification in the symplectic context. 

\begin{definition}{\cite{MR4423269}} \label{definition: realform}
Let $(M_\RR, \omega_\RR)$ be a smooth symplectic manifold of dimension $2n, \, n \ge 1$. A \emph{symplectic complexification} $(M, \omega, \tau)$ of $(M_\RR, \omega_\RR)$ is a holomorphic symplectic manifold $M$ of complex dimension $2n$, together with a holomorphic symplectic form $\omega$, an antiholomorphic involution $\tau \colon M \to M$ and a smooth map $\jmath \colon M_\RR \to M$, satisfying
\begin{enumerate}[label=(\roman*)]
	\item		The map  $\jmath$ is a proper embedding. 
	\item		The image of $M_\RR$ is the fixed-point set of $\tau$.
	\item		The pullback $\jmath^\ast \omega$ coincides with $\omega_\RR$. 	
	\item		The pullback $\tau^\ast \omega$ is the complex-conjugated $\overline{\omega}$.
\end{enumerate}
\end{definition}

\begin{remark}
    (a) Property (i) implies that $\jmath(M_\RR)$ is closed in $M$, thus we may identify $M_\RR$ with $\jmath(M_\RR)$ as a totally real submanifold of maximal dimension in $M$. 

    (b) We also refer to $M_\RR$ as the \emph{real form} of $M$ (with respect to $\tau$).
\end{remark}

 Let $\aut_{\omega}(M)$ be the group of holomorphic symplectic automorphisms of $M$ and $\diff_{\omega_\RR}(M_\RR)$ the group of symplectic diffeomorphisms from $M_\RR$ onto itself. Also denote by $\taut_{\omega} (M)$ the group of holomorphic symplectic automorphisms of $M$ that preserve the real form, i.e.
 \[	\taut_{\omega} (M) = \{ \Phi \in \aut_{\omega}(M) : \Phi(M_\RR ) = M_\RR \}
 \]
By property (iii), $\restr{\Phi}{ M_\RR}$ lies in $\diff_{\omega_\RR}(M_\RR)$. 
Moreover, (iii) implies the real parts of $\tau$-compatible (cf.\ Definition \ref{def: real-tangential}) $\omega$-symplectic vector fields $V,W$ are real $\omega_\RR$-symplectic on $\camo^\RR$:
\begin{align*}
    \omega(V, W) = \jmath^\ast \omega ( \alpha^{-1}(V), \alpha^{-1}(W)) = \omega_\RR (\alpha^{-1}(V), \alpha^{-1}(W))
\end{align*}
where $\alpha \colon \T M \to \T^{1,0} M$ is the $\RR$-isomorphism connecting smooth vector fields with holomorphic vector fields.  

Let $H \in \holo(M)$ be a holomorphic function and $V$ the holomorphic Hamiltonian vector field associated to $H$, namely $\dif H = i_V \omega$. One may compare the following lemma with \cite[Theorem 1.3]{arathoon2023real}.
\begin{lemma} \label{lem:Rholo-vs-Rtang}
    Under the settings of Definition \ref{definition: realform}, $H$ is $\tau$-compatible up to a constant if and only if $V$ is $\tau$-compatible. 
\end{lemma}
\begin{proof}
    Applying $\tau^\ast$ to $\dif H = i_V \omega$ yields   
\begin{align*}
    \dif \tau^\ast H = \tau^\ast \dif H = \tau^\ast (i_V \omega) = i_{\tau^\ast V} \overline{\omega}
\end{align*}
where the last equality follows from (iv).
\end{proof} 

As at the beginning of Section \ref{hamdiffeo}, the notion of Hamiltonian isotopy can be defined for holomorphic symplectic manifold as well. If such a holomorphic isotopy $\Phi_t$ is in $\taut_{\omega}(M)$ for every $t$, then their restriction $\restr{\Phi_t}{M_\RR} = \jmath^\ast \Phi_t$ is a Hamiltonian isotopy on the real form $(M_\RR, \omega_\RR)$ because $\tau$-compatible Hamiltonian vector fields have $\tau$-compatible Hamiltonian functions.

\smallskip

\section{The real Calogero--Moser space} \label{section: realform}
Now, we choose a real form for the Calogero--Moser space $\camo$ by taking the antiholomorphic involution $\tau$ on $\mathcal{M} = \matnn \oplus \matnn \oplus \CC^n \oplus (\CC^n)^\ast$ 
\begin{align*}
	\tau (X, Y, v, w ) = ( X^*,  Y^*, i w^*, i v^* )
\end{align*}
 The fixed-point set of $\tau$ is
\begin{align*}
	\mabove^{\tau}= \{ (X, Y, v, w) \in \mabove : X^* = X, Y^* = Y, v = i w^* \}
\end{align*}
In particular, the first two components consist of Hermitian matrices. 

It is straightforward to check for $g \in \gl_n(\CC)$,
\begin{align} \label{equation: tau-orbit}
	\tau ( g \cdot (X, Y, v, w) ) = g' \cdot \tau (X, Y, v, w)
\end{align}
where $g' = (g^\ast)^{-1}$. Thus the $\gl_n(\CC)$-orbit through $z \in \mabove$ is being taken to an orbit through $\tau(z)$. On the other hand, if $(X, Y, v, w)$ lies in $\mu^{-1} ( i I_n )$ then $\tau (X, Y, v, w)$ is contained in $\mu^{-1}( i I_n)$
\[
	[ X^\ast, Y^\ast ] + (i w^\ast) (i v^\ast) = X^\ast Y^\ast - Y^\ast X^\ast - w^\ast v^\ast =  i I_n
\] 
Therefore $\tau$ induces an antiholomorphic involution on $\camo$, which we also denote by $\tau$. 

On the other hand, the original setup in \cite{MR478225} is compatible with this choice of $\tau$. The real Calogero--Moser space was constructed as symplectic reduction over the unitary group in \cite[pp. 491--494]{MR478225}, starting with 
\begin{align*}
	\mabove_h = \mathfrak{h}(n) \oplus \mathfrak{h} (n) \oplus \CC^n
\end{align*}
where $\mathfrak{h}(n)$ denotes the $\RR$-vector space of Hermitian square matrices of size $n$. As a real vector space, $\mabove_h$ is of dimension $2n(n+1)$ and is isomorphic to $\mabove^{\tau} \subset \mabove$ under $(A, B, a) \mapsto ( A, B, a, i a^*)$.

The unitary group $\mathrm{U}(n)$ acts on $\mabove_h$ as
\begin{align*}
	u \cdot (A, B, a ) = ( u A u^*, u B u^*, u a ), \, \, u \in \mathrm{U}(n)
\end{align*}
which is the restriction of the $\gl_n (\CC)$-action on $ \mabove^{\tau} \cong \mabove_h$. 

Let $\mathfrak{u}(n)$ be the Lie algebra of the unitary group $\mathrm{U}(n)$, which consists of skew-Hermitian matrices. Take the real moment map 
\begin{align*}
	\mu_\RR \colon  \mabove_h \to \mathfrak{u}(n), \quad
	(A, B, a) \mapsto [A, B] + i a a^*
\end{align*}
which is the restriction of $\mu$ on $\mabove^{\tau}$. Denote by $\omega_\RR$ the pullback of the holomorphic form $\omega$ to $\mabove^{\tau}$ under the inclusion $\mabove^{\tau} \hookrightarrow \mabove$, which is a real symplectic form in the following real coordinates on $\mabove^\tau$
\begin{align*}
	&X_{jj}, Y_{jj}, \re (v_j),  \im (v_j), \ j = 1, \dots, n, \\
	&\re (X_{jk}),  \re (Y_{kj}),  \mathrm{Im} (X_{jk}),  \im (Y_{kj}), \ j < k
\end{align*}
with $j < k$. On the fixed-point set $\mabove^\tau$ 
\[
	\dif v_j \wedge \dif w_j = i \, \dif v_j \wedge  \dif \overline{v_j} = 2 \, \dif  \re (v_j) \wedge \dif \im (v_j)
\]
Then for $( A, B, a, i a^\ast) \in \mabove^\tau$ the real symplectic form $\omega_\RR$ is
\begin{align*}
	 \omega_\RR = 
	& \sum_j \left( \dif A_{jj} \wedge \dif B_{jj} + 2  \dif \re(a_j) \wedge \dif \im (a_j) \right) \\
	&+ 2 \sum_{j < k} \left( \dif \re (A_{jk}) \wedge \dif \re (B_{kj})  - \dif \im (A_{jk}) \wedge \dif \im (B_{kj}) \right)  \nonumber
\end{align*}
The $\gl_n(\CC)$-invariance of $\omega$ implies the $\un(n)$-invariance of $\omega_\RR$, which induces a real symplectic form on the quotient over $\un(n)$. 
With the symplectic form $\omega_\RR$ and the moment map $\mu_\RR$, we have the following. 

\begin{definition} \label{definition: realCM}
	The \emph{real Calogero--Moser space} $(\camo^\RR, \omega_\RR)$ is the real symplectic reduction $\mu_\RR^{-1} ( i I_n) / \un (n)$. 
\end{definition}

\begin{lemma} \label{lemma: symp=ham}
	The real Calogero--Moser space $\camo^\RR$ is simply connected.
\end{lemma}
\begin{proof}
	Since any Hermitian matrix is diagonalizable by conjugation of the unitary group, we can consider $A$ in diagonal form with decreasing diagonal entries $x_1, x_2, \dots, x_n$. 
	Recall that $(A, B, a) \in \caboveFR$ satisfies $[A, B] + iaa^\ast = i I_n$. As the commutator has zeros on the diagonal, $a$ is of the form $(e^{i \theta_1}, \dots, e^{i \theta_n})$. Use 
\[	u = \mathrm{diag}(e^{-i \theta_1}, \dots, e^{-i \theta_n}) 	
\]
to reduce all components of $a$ to be one. Then the off-diagonal entry $b_{jk}$ of $B$ satisfies  
\[ b_{jk} (x_j - x_k) + i = 0   \]
making $x_j \neq x_k$ and $b_{jk}= -i / (x_j - x_k)$,  
while the real diagonal entries $y_1, \dots, y_n$ of the Hermitian matrix $B$ remain free. We have
\[
	(x_1, \dots, x_n, y_1, \dots, y_n)
\]
with $x_1 > x_2 > \dots > x_n$ is a system of global coordinates for $\camo^\RR$, which is thus diffeomorphic to $\RR^{2n}$. 
\end{proof}

With the notation $\caboveFR = \mu_\RR^{-1} ( i I_n), \, \caboveF = \mu^{-1} ( i I_n)$ we have
\[
\begin{CD}
	\caboveF @>\tau>>	 \caboveF \\
	@VVV 			 @VVV\\
	\camo	@>\tau>>	\camo
\end{CD}
\quad \quad \quad
\begin{CD}
	\caboveFR @>j>>	 \caboveF \\
	@VVV 			 @VVV\\
	\camo^\RR	@>\jmath>>	\camo
\end{CD}
\]
The map $\jmath$ takes an orbit $\un(n) \cdot z$ with $z \in \caboveFR$ to the orbit $\gl_n(\CC) \cdot z$, and $j$ is the inclusion map.  
The first diagram commutes due to Equation \eqref{equation: tau-orbit}, while the second commutes naturally. To simplify the notation we use $\un = \mathrm{U}(n)$ and $\gln = \gl_n(\CC)$.

\begin{remark}
	In general, see \cite[Remark 2.1]{arathoon2023real}, the real quotient space is only the union of some components of the conjugation-fixed part of the complex quotient. 
	For the Calogero--Moser space, we have the simple situation $\jmath(\camo^\RR) = \camo^\tau$. 
\end{remark}

\begin{proof}[Proof of Theorem \ref{theorem: realform}]
We verify the conditions of Definition \ref{definition: realform}. Condition (iii) follows from the definition of $\omega_\RR$ as the pullback of $\omega$ under $\jmath$. 

For (iv), the pullback $\tau^* \omega$ is the complex conjugated form
\begin{align*}
	\tau^* \omega &= \tau^* \tr (\dif X \wedge \dif Y + \dif v \wedge \dif w)  \\
	&= \tr ( \dif X^* \wedge \dif Y^* + i^2 \dif w^* \wedge \dif v^*)  \\
	&= \tr ( \dif \overline{X} \wedge \dif \overline{Y} + \dif \overline{v} \wedge \dif \overline{w} )  \\
	&= \overline{\omega}
\end{align*}

To show (ii) that $\jmath(\camo^\RR)$ is the fixed-point set of $\tau$, let $g \in \gln, z \in \caboveFR$. By Equation \eqref{equation: tau-orbit}
\[	\tau ( g \cdot z ) = g' \cdot \tau(z) = g' \cdot z \in \gln \cdot z
\] 
thus the orbit $\gln \cdot z$ is stable under $\tau$ and  $\jmath(\camo^\RR)$ is a subset of $\camo^\tau$. 

Conversely, we show that when a $\gln$-orbit in $\caboveF$ is stable under $\tau$, then it meets $\caboveFR$ at one $\un$-orbit. 
Let $z_0$ be a point in a $\tau$-stable $\gln$-orbit. Take the $\un$-invariant function
\[
	p_z \colon \gln \to \RR_{\ge 0}, \quad  g \mapsto \| g \cdot z \|^2
\]
where 
\begin{align*}
	\| (X, Y, v, w) \|^2 &= \| X \|^2 + \| Y \|^2 + \|v \|^2 + \| w \|^2 \\
	&=  \tr X X^\ast + \tr Y Y^\ast + \tr v v^\ast + \tr w^\ast w
\end{align*}
Clearly $ \| \tau(z) \| = \| z\|$. By geometric invariant theory, see \cite[\S 3.1]{MR1711344}, $p_z$ is convex on the double coset $\un \backslash \gln / \gln_z$ and attains minimum exactly when $\mu_1(g \cdot z) = 0$, where 
\[
	\mu_1 (X, Y, v, w) = \frac{1}{2} \{ [X, X^\ast] + [Y, Y^\ast] + v v^\ast - w^\ast w \}
\]
Since the $\gln$-action is free, the isotropy group $\gln_z$ is trivial. 
Recall from \cite[\S 8]{MR1626461} that the complex space $\camo$ is homeomorphic to the hyperk\"ahler quotient
\[
	\camo = \mu^{-1}(i I_n) / \gln  \cong ( \mu_1^{-1} (0) \cap \mu^{-1} ( i I_n ) ) / \un
\]
Thus each orbit $\gln \cdot z$ meets the $\un$-stable set $\mu_1^{-1}(0)$ at exactly one $\un$-orbit, where $p_z$ attains minimum. 

By $\caboveFR \subset \mu_1^{-1}(0)$, we consider $z_0 = (X, Y, v, w)  \in  \mu_1^{-1} (0) \cap \caboveF$. Since the orbit $\gln \cdot z_0$ is $\tau$-stable, there exists $h_0 \in \gln$ such that 
\[
	\tau (z_0) = h_0 \cdot z_0
\]
By the convexity of $p_{z_0}$ on $\un \backslash \gln$, $\mu_1(z_0) = 0$, and $ \| \tau(z_0) \| = \| z_0\|$, it follows that $h_0 \in \un$. Then 
\[
	z_0 = \tau (\tau (z_0)) = \tau (h_0 \cdot z_0) = h_0 \cdot \tau (z_0) = h_0^2 \cdot z_0
\]
where the third equality is due to Equation \eqref{equation: tau-orbit} and $h_0' = (h_0^\ast)^{-1} = h_0$. Recall that if $g \in \gln$ fixes a point $z \in \caboveF$, then $g$ is the identity \cite[Corollary 1.4]{MR1626461}. This implies that $h_0^2 = I_n$. In combination with $h_0 \in \un$, we see that $h_0$ is also Hermitian. Hence $h_0$ is of the form $u D u^\ast$ for some $u \in \un$, and $D = \mathrm{diag} (d_1, \dots, d_n)$ with entries either $1$ or $-1$. Replacing $z_0$ by $u^\ast \cdot z_0$, we may assume that $h_0 = D$. Then from $\tau(z_0) = D \cdot z_0$ we have 
\[
	X^\ast = D X D, \quad Y^\ast = D Y D
\]
which implies that $X$ and $Y$ are normal. Hence there exists $u_1 \in \un$ which diagonalizes $X$.  

Move to the point $z_1 = u_1 \cdot z_0 = (X_1, Y_1, v_1, w_1)$ on the same $\un$-orbit. Because $z_1 \in \caboveF$
\begin{align} \label{equation: z1rankcond}
	i I_n = [X_1, Y_1] + v_1 w_1 = [X_1, Y_1] + i v_1 v_1^\ast u_1 D u_1^\ast 
\end{align}
where the second equality follows from $h_1 = u_1 D u_1^\ast$ and
\[
	\tau(z_1) = h_1 \cdot z_1 \implies i v_1^\ast = w_1  u_1 D u_1^\ast
\]
Thanks to $X_1$ being diagonal, the commutator $[X_1, Y_1]$ has zeros on the diagonal. Comparing the $j$-th diagonal entry of Equation \eqref{equation: z1rankcond} yields  
\[
	i = i \, d_j \lvert (u_1^\ast v_1 )_j \rvert^2
\]
which implies that $d_j = 1$. Therefore, $h_0 = I_n$ and $\tau(z_0) = z_0$. Since $\tau$ commutes with the $\un$-action, the entire orbit $\un \cdot z_0$ is contained in $\caboveFR$. This shows that $\caboveF^\tau \subset \jmath(\camo^\RR)$. 

For (i): The injectivity of $\jmath \colon \camo^\RR \to \camo$ is equivalent to that each $\gln$-orbit through a point $z \in \caboveF^\tau$ meets $\caboveFR$ at exactly one $\un$-orbit. By Equation \eqref{equation: tau-orbit}, this condition implies that the $\gln$-orbit is $\tau$-stable. Thus the injectivity follows from the above discussion, but we also give a direct proof: 
We show that for $z_1, z_2  \in \caboveFR$ if their $\un$-orbits have empty intersection, then their $\gln$-orbits have empty intersection. Suppose otherwise, that $g \cdot z_1 = z_2$ for some $g \in \gln$. Equation \eqref{equation: tau-orbit} and $\tau$-invariance yield $g' \cdot z_1 = z_2$. By \cite[Corollary 1.4]{MR1626461}, $g^{-1} g' = I_n$. It follows that $g \in \un$ and $\jmath$ is injective. Note that in the more general setting of \cite[Theorem 3.2]{arathoon2023real}, the injectivity of $\jmath$ is also a consequence of the free $\gln$-action on $\caboveFR$.

To check that $\jmath$ is an immersion, we refer to \cite[Theorem 3.2]{arathoon2023real}. 

The properness of $\jmath$ follows from $\camo^\RR = \camo \cap \RR^q	$, see Remark \ref{remark: tauCq}. 
This shows (i) of Definition \ref{definition: realform}. 
\end{proof}

\section{$\tau$-symplectic density property} \label{section: tangencyfunction}
It is well-known that instead of the moment map $\mu \colon \mabove \to \matnn$ and the preimage $\mu^{-1}(i I_n)$ of a coadjoint-invariant point, it is equivalent to consider the moment map 
\[	\hat{\mu} \colon \matnn \oplus \matnn \to \mathfrak{sl}_n, \quad (X,Y) \mapsto [X,Y]		
\]
and take the complex Calogero--Moser space as the symplectic reduction $\hat{\mu}^{-1}( \gln \cdot \xi) / \gln$ along a coadjoint orbit $\gln \cdot \xi $, where any off-diagonal entry of $\xi$ is $-i$ and diagonal entries $0$. For the real counterpart, the real Calogero--Moser space $\camo^\RR$ is the quotient $\hat{\mu}_\RR^{-1} ( \un \cdot \xi ) / \un$ along the coadjoint orbit $\un \cdot \xi$ where the real moment map is the restriction of $\hat{\mu}$
	\[
		\hat{\mu}_\RR \colon \mathfrak{h}(n) \oplus \mathfrak{h}(n)  \to  \mathfrak{u}(n), \quad (A, B) \mapsto [A, B]
	\]
Let 
\begin{align*}
	\cabove = \hat{\mu}^{-1} ( i I_n) \subset \matnn \oplus \matnn, \quad \caboveR = \hat{\mu}_\RR^{-1} ( i I_n) \subset \mathfrak{h}(n) \oplus \mathfrak{h}(n)
\end{align*}
Then $\cabove$ is the variety of $\matnn \oplus \matnn \cong \CC^{2n^2}$ consisting of a pair of matrices $(X,Y)$ such that $\rank ([X,Y]- i I_n ) = 1$. We also take the conjugation $\tau$ on $\cabove$ as 
$\tau(X, Y) = (X^\ast, Y^\ast)$. This induces a conjugation on the $\CC$-algebra $\CC[\cabove]$ of regular functions such that on scalars it is the complex conjugation. As mentioned in Section \ref{section: prelim}, of special interest are the $\tau$-compatible functions, namely holomorphic functions $F \colon \cabove \to \CC$ satisfying $F \circ \tau = \overline{ F }$, which are in particular real-valued on the real form.

\begin{lemma} \label{lemma: realgenerators}
The $\RR$-algebra of real-valued $\un$-invariant algebraic functions on $\caboveR$ can be generated by $ \{ \tr A^j B^k: j+k \le n^2 \}$.
\end{lemma}
\begin{proof}
	Let $f$ be a real-valued $\un$-invariant algebraic function on $\caboveR$. Then $f$ is a real polynomial in the entries of $X$ and $Y$. Extend $f$ to a real complex algebraic function $F$ on $\cabove$
	\[	(X, Y) \mapsto  \frac{1}{2} ( f(X, Y) + \overline{ f (\tau(X,Y))} )
	\] 
	The unitary group $\un$ is a totally real submanifold of maximal dimension in $\gln$ and elements of $ \CC[\cabove]$ are holomorphic, hence $F$ is $\gln$-invariant. By Etingof and Ginzburg \cite[\S 11, p.~322, Remark (ii)]{MR1881922}, the $\CC$-algebra $\CC[\cabove]^{\gln}$ of $\gln$-invariants on $\cabove$ can be generated by $\{ \tr X^j Y^k: j+k \le n^2 \}$. Write $F$ in terms of these generators over $\CC$. Separating the scalars into real and imaginary parts and collecting terms, we get $F = F_1 + i F_2$, where $F_l,  l = 1, 2$ are in the $\RR$-algebra generated by these trace functions. From Equation \eqref{equation: tau-equiv}, the generators $ \tr X^j Y^k$ are real, thus $F_l$ is real-valued on $\camo^\RR$. Thus $f = \restr{F}{\camo^\RR} = \restr{F_1}{\camo^\RR}$. 
\end{proof}

\begin{remark} \label{remark: tauCq}
Since the $\CC$-algebra $\CC[\cabove]^{\gln}$ of $\gln$-invariants on $\cabove$ can be generated by $\{ \tr X^j Y^k: j+k \le n^2\}$, we choose a minimal generating set $\mathcal{G}$. Denote by $q$ the cardinality of $\mathcal{G}$, then the Calogero--Moser space $\camo$ is a smooth affine variety in 
\[	\CC^q \cong \spec \CC [\mathcal{G}]
\]
On the generating set $\mathcal{G}$, the conjugation $ \tau (X, Y) = ( X^\ast, Y^\ast ) $ on $ \cabove $ acts as complex conjugation due to the cyclicity of the trace
\begin{align} \label{equation: tau-equiv}
	\tau^\ast ( \tr X^j Y^k ) = \tr (X^\ast)^j (Y^\ast)^k = \tr \overline{Y}^k \overline{X}^j = \overline{\tr X^j Y^k}
\end{align}
Hence $\tau$  descends to the complex conjugate on the coordinates of $\CC^q$. The real Calogero--Moser space $\camo^\RR$, being the fixed-point set of $\tau$ in $\camo$, is indeed the intersection of $\camo$ with $\RR^q$. This also shows that $\camo^\RR$ is totally real in $\camo$. 
\end{remark}

\begin{proposition} \label{proposition: realADP}
    On the real Calogero--Moser space $( \camo^\RR, \omega_\RR)$, every real algebraic Hamiltonian vector field can be written as real Lie combination of complete algebraic Hamiltonian vector fields, each of them associated to a $\gln$-invariant function from
    \[  \mathcal{F} = \{  \tr Y, \tr Y^2, \tr X^3, (\tr X)^2 \}. \]
\end{proposition}
\begin{proof}
It suffices to consider the Hamiltonian functions. 
Since $\camo^\RR$ is a symplectic reduction, an algebraic Hamiltonian function $f$ corresponds to a $\un$-invariant $\hat{f}$ on $\caboveR$ and by Lemma \ref{lemma: realgenerators} the algebraic $\un$-invariants as a $\RR$-algebra are generated by $\{ \tr A^j B^k: j+k \le n^2 \}$. The computation in \cite{CaloSymplo} shows that any algebraic $\gln$-invariant is contained in the complex Lie algebra generated by the four functions in $\mathcal{F}$. 
We point out that in \cite{CaloSymplo} all formulae come with real coefficients, where the condition is that $[X,Y] + I_n$ has rank one. However we have a different condition, namely $[X, Y] - i I_n $ has rank one, which introduces an imaginary factor. 

For a real-valued algebraic $\un$-invariant $\hat{f}$ on $\caboveR$, we can first extend it to a function $F$ on $\cabove$ as in the proof of Lemma \ref{lemma: realgenerators}. Then write it as a complex Lie combination of functions in $\mathcal{F}$ and separate it into real and imaginary parts $F = F_1 + i F_2$. Fortunately, since $F$ and the generating functions $\tr X^j Y^k$ are real-valued on $\camo^\RR$, $F_2$ is zero on $\caboveR$. Hence $\hat{f}$ is in the real Lie algebra generated by $\mathcal{F}$. 
\end{proof}

Let $M$ be a complex manifold with a real form $(N, \tau)$. Recall from Lemma \ref{lem:realFLOWtang} that a holomorphic vector field $V$ on $M$ is $\tau$-compatible if $\tau_\ast V = \overline{V}$, which implies $\alpha^{-1}(V_p) = 2 \re (V_p) \in \T_p N$ for all $p \in N$, where
\[
    \alpha \colon \T M \hookrightarrow \T M \otimes_\RR \CC \to \T^{1,0}M, \, V \mapsto \frac{1}{2}( V - i JV)
\]
is the $\RR$-isomorphism between $\T M$ and $\T^{1,0} M$ (Definition \ref{def: real-tangential}). In particular the $\RR$-flow of $V$ preserves the submanifold $N$. In fact, the following version of density property is appropriate for Carleman approximation. 

\begin{definition} \label{def-tangential-dp}
    Let $(M_\RR, \omega_\RR)$ be smooth symplectic manifold and $(M, \omega, \tau)$ a symplectic complexification of $(M_\RR, \omega_\RR)$. We say that $(M, \omega, \tau)$ admits the \emph{$\tau$-symplectic density property}, if the complex Lie algebra generated by $\CC$-complete $\tau$-compatible holomorphic symplectic vector fields is dense in the Lie algebra of holomorphic symplectic vector fields with respect to the compact-open topology. 
\end{definition}

For the purpose of approximation in the next section, it suffices for $\camo$ to have this density property with respect to $\tau$. 

\begin{lemma} \label{lemma: tangentDP}
    The Calogero--Moser space $\camo$ admits the $\tau$-symplectic density property. 
\end{lemma}
\begin{proof}
Since the Lie generators in $\mathcal{F}$ are all $\tau$-compatible with respect to $\tau$ by \eqref{equation: tau-equiv}, the claim follows from Lemma \ref{lem:Rholo-vs-Rtang} and the symplectic density property of $\camo$. 
\end{proof}

The following is a parametric approximation on compact subsets for real symplectic vector fields on the real Calogero--Moser space.  

\begin{proposition} \label{proposition: approxRbyC}
    Let $Z$ be a compact subset of $\RR^{N}$ for some natural number $N$, $V_z$ a continuous family of smooth Hamiltonian vector fields on the real Calogero--Moser space $\camo^\RR$. For any $\varepsilon > 0, k \in \mathbb{N}$, and compact $K \subset \camo^\RR$, there exists a continuous family of complete holomorphic Hamiltonian vector fields $W_{z, 1}, \dots, W_{z, m}$ on the complex Calogero--Moser space $\camo$, and a real Lie combination $L( W_{z, 1}, \dots, W_{z, m} )$, all having real-time flows that preserve  $\camo^\RR$, such that
	\[  \| V_z - \alpha^{-1} (L( W_{z, 1}, \dots, W_{z, m} ) ) \|_{C^k(K)} < \varepsilon \]
    where $\alpha \colon \T M \to \T^{1,0}M$ is the $\RR$-isomorphism between $\T M$ and $\T^{1,0} M$.
\end{proposition}
\begin{proof}
Fix $z \in Z$. 
In order to approximate Hamiltonian vector fields in the $C^k$-norm, by $d H =  \omega_\RR ( V_H, \cdot )$ it suffices to approximate the Hamiltonian function $f_z$ of $V_z$ in the $C^{k+1}$-norm.

On the compact subset $K \subset \camo^\RR$ we approximate the smooth function $f_z$ in the $C^{k+1}$-norm by a $\tau$-compatible polynomial $P_z$, using the Weierstrass approximation theorem for compact subsets in $\RR^q$ (see e.g. \cite[\S 1.6.2]{MR0832683}). Since $\camo$ is a submanifold in $\CC^q$, approximation of functions on $\CC^q$ in the $C^{k+1}$-norm on $\CC^q$ implies the corresponding approximation on $\camo$. 
By Proposition \ref{proposition: realADP} the lift of $P_z$ corresponds to a real Lie combination of finitely many $U$-invariants from $\mathcal{F}$ on $\caboveR$, hence we can approximate $f_z$ in the $C^{k+1}$-norm by a real Lie combination $L( g_{z, 1}, \dots, g_{z, m} )$ of functions $g_{z, 1}, \dots, g_{z, m}$ which correspond to complete vector fields. 

Next consider the extensions of the $\un$-invariant functions to $\cabove$ by requiring $\gln$-invariance. Let $W_{z, 1}, \dots, W_{z, m}$ be the corresponding complete Hamiltonian vector fields on $\camo$. For each $j$, $W_{z,j}$ is the Hamiltonian vector field of a Hamiltonian function in $\mathcal{F}$. Since $\omega_\RR = \jmath^\ast \omega$ where $\jmath \colon \camo^\RR \hookrightarrow \camo$, the real-valued $\un$-invariant functions induce, after being extended by $\gln$-invariance to $\cabove$, $\tau$-compatible Hamiltonian vector fields. The $\tau$-compatibility is stable under real-scalar multiplication, summation and taking Lie-brackets, thus the real Lie combination $L( W_{z, 1}, \dots, W_{z, m} )$ remains $\tau$-compatible. 
\end{proof}

\section{Local approximation} \label{section: localVFnAUTO}

To transfer the approximation in Proposition \ref{proposition: approxRbyC} for vector fields to one for symplectic diffeomorphisms, we need some preparation. 

\subsection*{From vector fields to flow maps.} 
In Varolin's article \cite{MR1829353} on the density property, he used the notion of algorithm from the theory of ordinary differential equations to compute the flow of Lie combination of vector fields. To a vector field, an algorithm captures the linear part of its flow. We recall the motivation for using Lie combinations to define the density property. 
\begin{lemma} \label{lemma: algoLiecombi} 
	For vector fields $V$ and $W$ with flows $\Phi_t, \Psi_t$, respectively. 
	\begin{enumerate}[label=(\roman*)]
		\item		The composition $\Phi_t \circ \Psi_t$ is an algorithm consistent with $V+W$.
		\item		For $t >0$, $\Psi_{-\sqrt{t}} \circ \Phi_{- \sqrt{t}} \circ \Psi_{ \sqrt{t}} \circ \Phi_{\sqrt{t}}$ is an algorithm consistent with $[V, W]$.
	\end{enumerate}
\end{lemma}
Varolin's intension of introducing the density property was to generalize Runge approximations of holomorphic injections of a domain by holomorphic automorphisms in $\CC^n$ to Stein manifolds $M$. This is based on the above lemma and that 
\[
	\lim_{n \to \infty} (A_{t/n})^{\circ n} (x) = \Phi_t (x), \quad x \in M
\]
where $A$ is an algorithm consistent with the vector field whose flow is $\Phi_t$. A parametric version with extra uniform condition can be found in \cite[Theorem 3.3]{MR4423269}, while the same statement with approximation in the $C^k$-norm was proved in the unpublished Diplomarbeit of B. Sch{\"a}r \cite{SchaerB} and cited in \cite[Theorem 3.4]{MR4423269}. 

To apply the (symplectic) density property for approximation of holomorphic automorphisms, a further ingredient is that the nearness of vector fields in the $C^k$-norm implies the nearness of their flows in the $C^k$-norm. This is the content of \cite[Lemma 3.2]{MR4423269}.

 \subsection*{Mergelyan approximation on admissible sets.} 
 Recall that a Stein compact is a compact subset which admits a basis of open Stein neighborhoods. A compact $\mathcal{O}(X)$-convex subset $Z$ in a Stein space $M$ admits a basis of open Stein neighborhoods of the form
\begin{align*}
	\{ p \in M : |f_1(p)| < 1, \dots, |f_N(p)|<1 \}
\end{align*}
for some $f_1, \dots, f_N \in \mathcal{O}(M)$, hence it is Stein compact. 
\begin{definition} \label{definition: admissible}
    A compact set $S$ in a complex manifold $M$ is called \emph{admissible}, if it is of the form $S = K \cup Z$, where $K$ is a totally real submanifold (possibly with boundary), $S$ and $Z$ are Stein compacts. 
\end{definition}

The next theorem is a parametric version of \cite{MR4264040}*{Theorem 20}. This version with parameters was mentioned in \cite{MR4142485}*{\S 5} with the hint that it can be obtained from the nonparametric case by applying a continuous partition of unity (similar to the proof of the parametric Oka-Weil theorem), which was carried out in detail in \cite{giraldo2023}*{Lemma 4.10}. 

\begin{theorem}{\cite[Theorem 20]{MR4264040}} \label{theorem: parametric-admissible}
	Let $S = K \cup Z$ be an admissible set in a complex manifold $M$, with $K$ a totally real submanifold (possibly with boundary) of class $C^k$. Then for any $f \colon [0,1] \times S \to \CC$ with $ f_t \in C^k (S) \cap \mathcal{O}(Z)$ for $t \in [0,1]$ and $f$ continuous in $t$, there exists a $t$-family of sequences $f_{j, t} \in \mathcal{O}(S)$ such that
	\[	\lim_{ j \to \infty } \sup_{t \in [0,1]} \, \lVert f_{j, t} - f_t \rVert_{C^k (S) } = 0	\]
\end{theorem}

 \subsection*{Sublevel sets in $\camo$.}
For $p_0$ in $\camo^\RR$, the mapping
\[	\CC^q \to \CC^q, \quad p \mapsto \| \tau(p) - p_0 \|^2
\]
has a strictly positive definite Levi form, because by Remark \ref{remark: tauCq} the conjugation $\tau$ interchanges a coordinate of $\CC^q$ with its complex conjugation and is thus antiholomorphic. Next, an easy computation shows that the composition of a smooth strictly plurisubharmonic exhaustion function with an antiholomorphic diffeomorphism is strictly plurisubharmonic.

\begin{definition} \label{definition: psh}
	We take the strictly plurisubharmonic exhaustion function 
	\begin{align*}
		\rho \colon \camo \to \RR_{\ge 0}, \quad p \mapsto \| p - p_0 \|^2 + \| \tau(p) - p_0 \|^2
	\end{align*}
	where the norm $\| \cdot \|$ is the restriction of the standard norm of $\CC^q$. 
	Consider the closed sublevel set 
	\begin{align} \label{sublevelset}
		Z_R = \{ p \in \camo: \rho(p) \le R \}		
	\end{align}
	Denote by
	\begin{align*}
		Z_R^\RR = Z_R \cap \camo^\RR
	\end{align*}
	the intersection of $Z_R$ with the real part. 
\end{definition}

With these tools at hand, we show an approximation of Anders{\'e}n-Lempert type for symplectic diffeomorphisms of $\camo^\RR$ onto itself locally on Saturn-like sets. 
\begin{theorem} \label{theorem: LocalCarleman}
	Let $\varphi \colon [0,1] \times \camo^\RR \to \camo^\RR$ be a Hamiltonian isotopy and $R \ge 0$ such that $\restr{\varphi_t}{Z_R^\RR} = \id_{Z_R^\RR}$ for all $t$ in $[0,1]$.
	Then for any $k \in \NN$, any $b \in (0,R)$ and any $a > b$, there exists a sequence of holomorphic Hamiltonian isotopies $\Phi_{j} \colon [0,1] \times \camo \to \camo$ such that $\Phi_{j,t} \in \taut_{\omega}(\camo)$ for all $t$ in $[0,1]$ and 
	\begin{align}
		\| \Phi_{j,t} - \varphi_t \|_{C^k ( Z_{a}^\RR)} &\xrightarrow{j \to \infty} 0 	\\
		\| \Phi_{j,t} - \mathrm{id} \|_{C^k ( Z_{b})}		&\xrightarrow{j \to \infty} 0 	 \label{ApproxOnZb}
	\end{align}
	uniformly for $t$ in $[0,1]$.
\end{theorem}

\begin{proof}
	Let $K \subset \camo^\RR $ be a compact subset of class $C^{k+1}$ such that 
	\[    \bigcup_{ t \in [0,1] }  \varphi_t ( Z_{a}^\RR )  \ssubset \mathrm{Int} (K)
	\]
	 This condition suffices for our purpose of approximating $\varphi_t$. In Step 1 we switch to the smooth family of  vector fields $V_t$ generating $\varphi_t$ and approximate $V_t$ on $K$. By \eqref{equation: V_t}, to obtain approximation of flows from approximation of vector field, it suffices to impose the inclusion that the image of $Z^\RR_a$ under $\varphi_t$ for all $t$ is contained in $K$. On this compact $K$ we invoke Proposition \ref{proposition: approxRbyC} to approximate smooth vector fields by Lie combinations of complete $\tau$-compatible vector fields.

	\textbf{Step 1}: By assumption $\varphi_t$ is a Hamiltonian isotopy, thus there exist a family of Hamiltonian functions $P_t$ and the corresponding family of Hamiltonian vector fields $V_t$ such that 
	\begin{align} \label{equation: V_t}
		\frac{\dif } {\dif t} \varphi_t = V_t \circ \varphi_t , \quad t \in [0,1]		
	\end{align}
	By the second assumption on $\varphi_t$ we can extend $V_t$ to be identically zero in an open neighborhood of $Z_{b}$. 
	
	For a fixed $m \in \NN$, fixed $t \in [0,1]$, and $j \in \{0, 1, \dots, m-1\}$, denote by $\varphi^{(j)}_s$ the time-$s$ flow map of the time independent vector field $ V_{jt/m}$. We can choose $m$ large enough so that $\varphi^{(j)}_{t/m}$ is well-defined. 
	Now we take
	\begin{align*}
		\varphi^m_t : =  \varphi^{(m-1)}_{t/m} \circ \varphi^{(m-2)}_{t/m} \circ \cdots \circ  \varphi^{(1)}_{t/m} \circ  \varphi^{(0)}_{t/m} 
	\end{align*}

	By \cite[Corollary 3.5]{MR4423269} the composition $\varphi^{m}_{t}$ of flow maps converges to $\varphi_t$ uniformly in the $C^k$-norm in $p$ and uniformly in $t$ as $m$ tends to infinity. Therefore for a large enough $m$, it will be enough to show that each $t$-parameter family of flow maps $\varphi^{(j)}_s$ can be approximated in the $C^k$-norm in $p$ by a $t$-parameter family of holomorphic automorphisms in $\taut_{\omega}(\camo)$ uniformly in $(t,s)$. 
Then by \cite[Lemma 3.2]{MR4423269}, that approximation for vector field implies approximation for its flow, we may simply approximate the $t$-parameter family of Hamiltonian vector fields $V_{jt/m}$ uniformly in the $C^k$-norm in $p$ and uniformly in $t$, by complete $\tau$-compatible vector fields. 
	
	\textbf{Step 2}: Here we will use Proposition \ref{proposition: approxRbyC} to approximate the real vector field $V_{jt/m}$ on $K$ with the extra requirement from \eqref{ApproxOnZb} that the approximation is uniformly close to zero on $Z_b^\RR$. Let $P_{j, t} = P_{j t /m}$ denote the Hamiltonian function for $V_{jt/m}$ on $\camo^\RR$ and extend $P_{j, t}$ to be zero on $Z_b$. This extension is allowed since $\restr{\varphi_t}{Z_R^\RR} = \id_{Z_R^\RR}$ for all $t$ in $[0,1]$ and $b < R$. 
	
	To find a $\tau$-compatible function on $\camo$ approximating $P_{j, t}$, we recall (from Remark \ref{remark: tauCq}) that 
\[
	\camo^\RR = \camo \cap \RR^q \subset \CC^q
\] 
Since $\camo$ is an affine variety in $\CC^q$, the $\mathcal{O}(\camo)$-convex subset $Z_b \subset \camo$ is polynomially convex in $\CC^q$ by Cartan's Theorem B. 
By a result of Chirka and Smirnov \cite[Theorem 2]{MR1155560}, the union of a compact set in $\RR^q$ and a polynomially convex compact set in $\CC^q$, which is symmetric with respect to $\RR^q$, is polynomially convex. This yields the polynomial convexity of $K \cup Z_b$, and therefore its Stein compactness. 

Since $K$ is a totally real submanifold of class $C^{k+1}$ and $K \cup Z_b$ is an admissible set in $\CC^q$ (cf.\ Definition \ref{definition: admissible}), we can use the above parametric version of \cite[Theorem 20]{MR4264040} to approximate $\restr{P_{j, t}}{K \cup Z_b}$ in the $C^{k+1}$-norm by a function $P'_{j,t}$ which is holomorphic in a neighborhood of $K\cup Z_b \subset \CC^q$. Then by the Oka-Weil theorem with parameters (e.g.\ see \cite[Proposition 2.8.1]{MR3700709}), there is a $t$-parameter family of holomorphic polynomials $Q_{j,t}$ which approximates  $P'_{j,t}$ on $K \cup Z_b$. Replacing $Q_{j,t}$ by 
\begin{align*}
	\frac{1}{2} \left( Q_{j,t}  + \overline{ \tau^\ast Q_{j,t}  } \right) 
\end{align*}
makes it real-valued on $\camo^\RR$. By Cartan's Theorem B, restricting $Q_{j, t}$ to the affine variety $\camo$ yields the desired $\tau$-compatible function. 
Therefore, the associated Hamiltonian vector field $V_{j,t}$ for $Q_{j,t}$ approximates the real Hamiltonian vector field $V_{jt/m}$ for $P_{j,t}$ on $K$ and zero on $Z_b$. As in the proof of Proposition \ref{proposition: approxRbyC}, there are complete $\tau$-compatible Hamiltonian vector fields $W_{t, 1}, \dots, W_{t, M}$ on $\camo$, such that a Lie combination $L( W_{t, 1}, \dots, W_{t, M} )$ of them approximates $V_{j,t}$.

	\textbf{Step 3}: Finally, let $\psi^{t, a}_s$ be the time-$s$ flow map of $W_{t, a}$ for $a = 1, \dots, M$. Let $\Psi^{t, j}_s$ be an algorithm consistent with (cf.\ Lemma \ref{lemma: algoLiecombi}) the Lie combination $L( W_{t, 1}, \dots, W_{t, M} )$.  To see how the approximation operates on the level of the automorphisms and to keep the notation manageable, we consider the simple case when $L( W_{t, 1}, \dots, W_{t, M} )$ is the sum. Then
	\[		\Psi^{t, j}_s = 	\psi^{t, M}_s \circ \dots \circ  \psi^{t, 1}_s	\] 
	is an algorithm consistent with the sum $W_t = \sum_a W_{t, a}$, whose time-$(t/m)$ flow map can be approximated by 
	\begin{align*}	
		( \Psi^{t,j}_{t/(m l) } )^{l}, \quad   l \to \infty	
	\end{align*} 
	in the $C^k$-norm according to \cite[Theorem 3.4]{MR4423269}. Since $W_t$ approximates $V_{j,t}$, which in turn approximates $V_{jt/m}$ on $K$ and zero on $Z_b$. By \cite[Lemma 3.2]{MR4423269} we have that on $Z^\RR_a$ the flow of $W_t$ approximates the flow $\varphi^{(j)}_{t/m}$ of $V_{jt/m}$. Therefore 
	\begin{align*}	
		( \Psi^{t,j}_{t/(m l) } )^{l}    \xrightarrow{l \to \infty}	\varphi^{(j)}_{t/m}
	\end{align*} 
	which is the missing link in Step 1.  
\end{proof}

\begin{remark}
	There is constraint only on $b$ since it needs to be smaller than $R$ for which $\varphi_t$ is already the identity on the real part $Z_R^\RR$. However we may choose $a$, which determines where the automorphisms approximate the given family of diffeomorphisms, as large as the situation requires. 
\end{remark}

\section{Global approximation} \label{section: globalAUTO}

The next statement is the counterpart of \cite[Theorem 1.2]{MR4423269} and likewise is a consequence of Theorem \ref{theorem: LocalCarleman}. Namely we will construct the global approximation by applying local approximation successively. 
\begin{theorem} \label{theorem: HamCarleman}
	Let $\varphi \in \diff_{\omega_\RR}(\camo^\RR)$ be a symplectic diffeomorphism which is smoothly isotopic to the identity. Then for any positive continuous function $\epsilon$ on $\camo^\RR$, there exists a holomorphic symplectic automorphism $\Phi \in \taut_{\omega}(\camo)$ such that
	\[	\| \Phi - \varphi \|_{C^k(p)} < \epsilon (p)	\]
	for all $p$ in $\camo^\RR$.  
\end{theorem}

\begin{remark}
(i) The key idea parallels the proof of the classical Carleman approximation on $\RR \subset \CC$ by entire functions, see for example in \cite[Theorem 8]{MR4264040}. Theorem \ref{theorem: LocalCarleman} ensures the local approximation and requires the holomorphic convexity of Saturn-like sets, and a combination with the push-out method in \cite{MR1760722} delivers the wanted holomorphic automorphism in the limit. 

(ii) To use Theorem \ref{theorem: LocalCarleman}, we need an isotopy $\psi_t$ of Hamiltonian diffeomorphisms \emph{equal} to the identity on a compact subset $K \subset \camo^\RR$. In the proof to follow, we usually end up with an isotopy $\phi_t$ only \emph{approximating} the identity on $K$. Thus an extra interpolation is necessary. The goal is to construct a family $\psi_t$ which is equal to $\phi_t$ on the complement of an open neighborhood containing $K$. To achieve this, we multiply the time dependent Hamiltonian function $H_t$ of $\phi_t$ with a cutoff function $\gamma$, which is 0 inside $K$ and 1 outside a neighborhood $U_1$ of $K$. The new Hamiltonian function $\gamma H_t$ induces a slightly different vector field $W_t$ with flow map $\psi_t$.

Moreover, we assume that $\phi_t$ is close to the identity on $U_2$, equivalently $V_t$ is close to zero on an open set $U_2$ containing $U_1$. This additional buffer zone $U_2 \setminus U_1$ is used to secure that the new flow $\psi_t$ is equal to $\phi_t$ on the complement of $U_2$: A point $p$ close to $U_2$ might flow along $V_t$ into $U_2$, but since $V_t$ is by assumption small on $U_2$, the trajectory will stay outside $U_1$, where the function $\gamma$ is 1 and the flow is simply $\phi_t$. Inside $K$, $\psi_t$ is the identity by the definition of $\gamma$. 

The Hamiltonian vector field $W_t$ for $\gamma H_t$ is associated to the 1-form
\[
    \dif (\gamma H_t) = \gamma \dif H_t + H_t \dif \gamma
\]
The flow map $\psi_t$ exists for time $t \in [0,1]$ since $H_t$ is close to $0$ on $U_1 \setminus K$.

\end{remark}

\begin{proof}\footnote{For automorphisms, uppercase Greek letters denote holomorphic symplectic automorphisms, while lowercase symplectic diffeomorphisms. }
	By Lemma \ref{lemma: symp=ham} every symplectic isotopy in $\diff_{\omega_\RR}(\camo^\RR)$ is a Hamiltonian isotopy. Thus there exists a smooth Hamiltonian isotopy $\alpha_t$ in $\aut_{\omega_\RR}(\camo^\RR), t \in [0,1]$ with $\alpha_0 = \mathrm{id}$ and $\alpha_1= \varphi$. 

	Assume that we have a holomorphic symplectic automorphism $\Phi_j$ with $\Phi_j ( \camo^\RR ) = \camo^\RR$, a Hamiltonian isotopy $\psi_{j,t}$ of $\camo^\RR$, real numbers $R_j, S_j$ with $R_j \ge j-1, S_j \ge R_j + 1$ such that
	\smallskip
	\begin{enumerate}
		\item[($1_j$)] 	The image of $ Z_{R_j} $ under $\Phi_j$ is contained in $Z_{S_j}$. 
		\item[($2_j$)]	For $j \ge 2$ 
					\[	 \| \Phi_j - \Phi_{j-1} \|_{C^k(Z_{R_{j-1}})} < \varepsilon_j  \] 
		\item[($3_j$)]	At time zero $\psi_{j,0} \colon \camo^\RR \to \camo^\RR$ is the identity.
		\item[($4_j$)]	On $Z^\RR_{S_j }$ we have $\psi_{j,t}$ is the identity for $t$ in $[0,1]$. 
		\item[($5_j$)]	For all $p$ in $\camo^\RR$
					\[	\| \psi_{j, 1} \circ \Phi_j - \varphi \|_{C^k(p)} < \epsilon (p)	\]
	\end{enumerate}
	The induction hypothesis is that given any positive $\varepsilon_j$, $(1_j) -(5_j)$ can be realized.  

\textbf{Induction base}: For $j = 1$ take \textcolor{blue}{$R_1 = 0$}. In this case we have 
	\[
		Z_0 = Z_0^\RR =  \{ p_0 \} \subset \camo^\RR
	\]
	from \eqref{sublevelset} and by the choice of the strictly plurisubharmonic exhaustion function. 
	Choose \textcolor{blue}{$S_1$} so that 
	\begin{align} \label{equation: S_1}
		\alpha_1 \left( Z_{ 1}^\RR \right) \ssubset Z_{S_1}^\RR
	\end{align}
	Choose $r > S_1 $. Then by \eqref{equation: S_1}
	\begin{align}
		Z_{0}^\RR \subset Z_1^\RR  \ssubset (\alpha_1)^{-1} \left( Z_{S_1}^\RR \right) \subset (\alpha_1)^{-1} \left( Z_{r}^\RR \right)	\label{equation: R_1}	
	\end{align}
	Using Theorem \ref{theorem: LocalCarleman} we get a holomorphic Hamiltonian isotopy $A_t \in \taut_\omega(\camo)$ approximating $\alpha_t$ on the compact subset 
	\[
		\bigcup_{ t \in [0,1]} (\alpha_t)^{-1} \left( Z_{r + 3 }^\RR \right)
	\]
	of $\camo^\RR$. Since $A_t$ approximates $\alpha_t$ on this compact, we have 
	\begin{align}
		(A_t)^{-1} \left( Z_{r + 2}^\RR \right) 	\ssubset  	(\alpha_t)^{-1} \left( Z_{r + 3}^\RR \right)  \label{equation: r_3}
	\end{align}
	In particular, $A_1$ approximates $\alpha_1$ on $A_1^{-1} \left( Z_{r + 2}^\RR \right) $, hence $\alpha_t \circ \restr{A_t^{-1}}{\camo^\RR}$ is close to the identity on $Z_{r + 2}^\RR$. Choose \textcolor{blue}{$\Phi_1 = A_1$}.

	To construct a Hamiltonian isotopy $\psi_{1,t}$ we interpolate between the identity on $Z_{r}^\RR $ and $\alpha_t \circ \restr{A_t^{-1}}{\camo^\RR}$ outside $Z_{r + 2}^\RR$. More precisely, let $Q_t$ be the Hamiltonian function associated to the Hamiltonian isotopy $\alpha_t \circ \restr{A_t^{-1}}{\camo^\RR}$. Fix $\psi_{1,t}$ to be the identity on $Z_{r}^\RR$ by multiplying $Q_t$ with a cutoff function $\gamma$, which is zero on $Z_{r}^\RR$ and one outside $ Z_{r + 1}^\RR $.  
	Namely, we take \textcolor{blue}{$\psi_{1,t}$} to be the isotopy of symplectic diffeomorphisms which comes from the Hamiltonian function $\gamma Q_t$. 
	
	By the choice of $r$ we have that $Z_{S_1}^\RR$ is contained in $Z_{r}^\RR$, which implies that $\psi_{1,t}$ is the identity on $Z_{S_1}^\RR$ for all $t$ in $[0,1]$. This shows $(4_1)$. 
	\smallskip 
	
	To see that $(5_1)$ is satisfied, let $p \in \camo^\RR$ and consider separately
	\begin{enumerate}[label=(\roman*)]
		\item		$p \in A_1^{-1} \left( Z_{r}^\RR \right)$: $\psi_{1,1}$ is the identity at $A_1 (p)$ and $A_1$ approximates $\alpha_1$ by \eqref{equation: r_3}. 
		\smallskip
		\item		$p \notin A_1^{-1} \left( Z_{r + 2}^\RR \right)$: $\psi_{1,1}$ is equal to $\alpha_1 \circ A_1^{-1}$ at $A_1 (p)$ by the choice of $\gamma$.
		\smallskip
		\item		$p \in A_1^{-1} \left( Z_{r + 2}^\RR \setminus Z_{r}^\RR \right)$: $A_1$ approximates $\alpha_1$ by \eqref{equation: r_3} and $\psi_{1,1} \circ \alpha_1$ is the interpolation between $\mathrm{id} \circ \alpha_1$ 				and $\alpha_1 \circ A_1^{-1} \circ \alpha_1$. Here $\alpha_1 \sim A_1$ implies $A_1^{-1} \circ \alpha_1 \sim \mathrm{id}$.  
	\end{enumerate}
	
	Last, let us check that $(1_1)$ holds. By the choice of $A_t$, we may assume that there exists a small positive $\delta$ which is less than one, such that 
	\[
		\Phi_1 \left( Z^\RR_{0} \right) \subset 	\alpha_1 \left( Z_{ \delta}^\RR \right) \subset \alpha_1 \left( Z_{ 1}^\RR \right) \ssubset Z_{S_1}^\RR 
	\]
	The first inclusion follows from the fact that $\Phi_1 = A_1$ approximates $\alpha_1$ by \eqref{equation: R_1} and the last inclusion is due to the choice of $S_1$ in $\eqref{equation: S_1}$. This concludes the induction base. 

\textbf{Induction step}: Take \textcolor{blue}{$R_{j+1} = S_j + 1$} and choose $a > \max \{ R_j + 1, R_{j+1} \}$ such that
	\begin{align}
		\Phi_j ( Z_{R_{j+1}} ) \ssubset	Z_a \label{equation: settinga}
	\end{align}	
	By Theorem \ref{theorem: LocalCarleman} there exists a holomorphic Hamiltonian isotopy $\Psi_{j,t}$ in $\taut_\omega(\camo)$, which approximates the identity near $Z_{S_j}$ and $\psi_{j,t}$ near $Z^\RR_{a + 2}$.  
	Thus 
	\[	
		\sigma_{j,t} = ( \restr{\Psi_{j,t}}{\camo^\RR} )^{-1} \circ \psi_{j,t}	
	\]
	approximates the identity on $Z_{a+2}^\RR$. Being the composition of two Hamiltonian isotopies, $\sigma_{j,t}$ is also a Hamiltonian isotopy. 

	Moreover take a cutoff function $\chi$ on $\camo^\RR$ such that it is zero on $Z^\RR_a$ and equal to one outside $Z_{a+1}^\RR$. Let $P_t$ be the Hamiltonian function associated to the Hamiltonian isotopy $\sigma_{j,t}$ and let $\tilde{\sigma}_{j,t}$ be the flow map of the vector field whose Hamiltonian function is $\chi P_t$. Then $\tilde{\sigma}_{j,t}$ is the identity on $Z^\RR_a$, close to the identity on $Z^\RR_{a+2}$, and equal to $\sigma_{j,t}$ outside $Z^\RR_{a+2}$.
	
	Then we have on $\camo^\RR$
	\begin{align}
		\Psi_{j,t} \circ \tilde{\sigma}_{j,t}  &\sim  \psi_{j,t}  \label{equation: psi_jt} \\ 
		(\Psi_{j,t} \circ \tilde{\sigma}_{j,t})^{-1} &\sim (\psi_{j,t})^{-1}  
	\end{align}	
	by the choices of $\Psi_{j,t}, \sigma_{j,t}, \tilde{\sigma}_{j,t}$.
	Next, choose \textcolor{blue}{$S_{j+1} > a $} so that 
	\begin{align} \label{equation: Sj+1FromPsi}
		\Psi_{j,1} ( Z_a ) \ssubset Z_{S_{j+1}} 	
	\end{align}
	Set $b = S_{j+1} + 1$ and pick $c$ and $d$ so that
	\begin{align} \label{equation: settingc}
		Z_{b+2}^\RR &\ssubset \varphi \left( Z_c^\RR \right) \\
		\Phi_j   \left(  Z_c^\RR \right) &\ssubset Z_d^\RR \label{equation: settingd1} \\
		Z_{c}^\RR &\ssubset \psi_{j,1} \left( Z_d^\RR \right) \label{equation: settingd2}
	\end{align}
	
	Next, apply Theorem \ref{theorem: LocalCarleman} to obtain an isotopy $\Sigma_{j,t} \in \taut_\omega(\camo)$, which approximate $\tilde{\sigma}_{j,t}$ on $Z^\RR_{d}$ and the identity on $Z_a$. Set \textcolor{blue}{$\Phi_{j+1} = \Psi_{j,1} \circ \Sigma_{j,1} \circ \Phi_j$}. The above choices of $c$ and $d$ allow us to approximate on $Z_c^\RR$
	\begin{align}
		\varphi \sim \psi_{j,1} \circ \Phi_j \sim \Psi_{j,1} \circ \tilde{\sigma}_{j,1} \circ \Phi_j \sim \Psi_{j,1} \circ \Sigma_{j,1} \circ \Phi_j  = \Phi_{j+1}
	\end{align}
	where the first approximation comes from $(5_j)$, the second by \eqref{equation: psi_jt}, and the third due to \eqref{equation: settingd1}. Combining this with \eqref{equation: settingc} we have
	\begin{align} \label{equation: lambda=one}
		Z_{b+2}^\RR  \ssubset \Phi_{j+1} \left( Z_{c}^\RR \right)
	\end{align}
	Furthermore take the Hamiltonian isotopy
	\[	
		\hat{\sigma}_{j,t} = \psi_{j,t} \circ  \restr {(\Psi_{j,t} \circ \Sigma_{j,t} )^{-1} }{\camo^\RR} 	
	\]
	and let $\hat{P}_t$ be the corresponding Hamiltonian function for $\hat{\sigma}_{j,t}$. Moreover let $\lambda$ be a cutoff function on $\camo^\RR$ such that it is zero on $Z^\RR_{b}$ and equal to one outside $Z^\RR_{b+1}$. Finally let \textcolor{blue}{$\psi_{j+1, t}$} denote the flow map of the vector field whose Hamiltonian function is $\lambda \hat{P}_t$. 	

	We go through the five conditions for the induction step: 
	
		($1_{j+1}$)	By the definition of $\Phi_{j+1}$ we have that
						\[
							\Phi_{j+1} (Z_{R_{j+1}}) = \Psi_{j,1} \circ \Sigma_{j,1} \circ \Phi_j (Z_{R_{j+1}})
						\]
						The claim follows from the fact that $\Sigma_{j,t}$ approximates the identity on $Z_a$ and from the choices of $a, S_{j+1}$ in \eqref{equation: settinga}, \eqref{equation: Sj+1FromPsi} respectively. 
		
		($2_{j+1}$)	Here we want to estimate 
						\[
							 \| \Phi_{j+1} - \Phi_{j} \|_{C^k(Z_{R_{j}})} =  \| \Psi_{j,1} \circ \Sigma_{j,1} \circ \Phi_j - \Phi_{j} \|_{C^k(Z_{R_{j}})} 
						\] 
						By $(1_j)$ we have 
						\[ 	
							\Phi_j(Z_{R_j}) \subset Z_{S_j}		
						\] 
						The estimate follows from the fact that $\Psi_{j,1}$ and $ \Sigma_{j,1}$ approximate the identity on $Z_{S_j}$.
		
		($3_{j+1}$)	A flow map at time zero is the identity.
		
		($4_{j+1}$)	Because $\psi_{j+1,t}$ is obtained by interpolating between the identity on $Z_b^\RR$ and $\hat{\sigma}_{j,t}$ outside $Z_{b+2}^\RR$, it follows that $\psi_{j+1,t} (p) = (p)$ for $p$ near $Z^\RR_{S_{j+1}} $ because $b = S_{j+1} + 1$.
		
		($5_{j+1}$)	To see that 
						\[
							\| \psi_{j+1, 1} \circ \Phi_{j+1} - \varphi  \|_{C^k(p)} < \epsilon (p)
						\]
						for all $p$ in $\camo^\RR$ we consider separately: 
			\begin{enumerate}[label=(\roman*)]
				\item		On the complement of $(\Phi_{j+1})^{-1} \left( Z^\RR_{c} \right)$, we have $\lambda = 1$ at $\Phi_{j+1}(p)$ thanks to \eqref{equation: lambda=one}.
						Then $ \psi_{j+1, 1}$ is 
						\[	
						\hat{\sigma}_{j,1} = \psi_{j,1} \circ ( \Psi_{j,1} \circ \Sigma_{j,1} )^{-1} 
						\]
						which together with $\Phi_{j+1} = \Psi_{j,1} \circ \Sigma_{j,1} \circ \Phi_j$ reduces this case to $(5_j)$. 
				\item		On $\Phi_j^{-1} \left( Z_d^\RR \right)$, we have by the choice of $\Sigma_{j,1}$ that 
						\begin{align}\label{equation: phi(j+1)}
						\Phi_{j+1} = \Psi_{j,1} \circ \Sigma_{j, 1} \circ \Phi_j  \sim \Psi_{j,1} \circ \tilde{\sigma}_{j, 1} \circ \Phi_j  \sim \psi_{j,1} \circ \Phi_j	
						\end{align}
						which approximates $\varphi$ by the induction hypothesis $(5_j)$. Moreover, this implies that $\Psi_{j,1} \circ \Sigma_{j, 1} \sim \psi_{j,1}$ on $Z_d^\RR$. Then on 
						\begin{align*} 
							\Phi_{j+1} \circ \Phi_j^{-1} \left( Z_d^\RR \right)  = \Psi_{j,1} \circ \Sigma_{j, 1} \left( Z_d^\RR \right)	
						\end{align*}
						we have
						\begin{align*}
							\hat{\sigma}_{j,1}  &= 	\psi_{j,1} \circ \left(  \Psi_{j,1} \circ \Sigma_{j,1} \right)^{-1}  \\
							&\sim \psi_{j,1} \circ \left(\psi_{j,1} \right)^{-1}  =  \mathrm{id}	
						\end{align*}
						Since 
						\[
							Z_{b+2}^\RR \subset \Phi_{j+1}(Z_c^\RR) \subset \Phi_{j+1} \circ \Phi_j^{-1} (Z_d^\RR) 
						\]
						this justifies our interpolation $\psi_{j+1, t}$ between the identity on $Z_b^\RR$ and $\hat{\sigma}_{j,t}$ on $Z_{b+2}^\RR$. Hence $\psi_{j+1,1}$ is either close to the identity (reducing to the above \eqref{equation: phi(j+1)}) or equal to 
						\[	\hat{\sigma}_{j,1} = \psi_{j,1} \circ  \restr {(\Psi_{j,1} \circ \Sigma_{j,1} )^{-1} }{\camo^\RR} 
						\]
						which composed with $\Phi_{j+1}$ again reduces it to $(5_j)$.
				\item		It suffices to consider the above two cases, because by the choice of $d$ in \eqref{equation: settingd2} and the approximation $ \psi_{j,1} \sim \Psi_{j,1} \circ \tilde{\sigma}_{j,1} $ in \eqref{equation: psi_jt}
						\[	
							Z_{c}^\RR \ssubset \Psi_{j,1} \circ \Sigma_{j,1} \left( Z_d^\RR \right) = \Phi_{j+1} \circ \Phi_j^{-1} \left( Z_d^\RR \right)		
						\]
						it follows
						\[	
							(\Phi_{j+1})^{-1} \left( Z_{c}^\RR \right)  \ssubset (\Phi_j)^{-1} \left( Z_d^\RR \right)	
						\]
		\end{enumerate}
	This completes the induction step. 
	
	\smallskip
	Concluding, to see that $\lim_{j \to \infty} \Phi_j$ exists, we underline the connection between the above induction and the push-out method, see e.g.\ \cite[Proposition 5.1]{MR1760722}. The holomorphic symplectic automorphism $\Phi_{j+1} \circ \Phi_j^{-1} = \Psi_{j,1} \circ \Sigma_{j,1}$ approximates the identity on $Z_{R_j}$. Since $R_j \ge j -1$, the sublevel set $Z_{R_j}$ exhausts $\camo$ in the limit. Additionally we have $R_{j+1} = S_j + 1 > R_j + 2$. Choose $\{ \varepsilon_j \}_{j \in \NN} \subset \RR_{>0}$ to have finite sum such that
	\[	0 < \varepsilon_{j+1} < \mathrm{dist} ( Z_{R_{j}} , \camo \setminus Z_{R_{j+1}} )
	\]
	where $\mathrm{dist}(\cdot, \cdot)$ is the distance function on $\CC^q$ restricted to $\camo$. 
	Therefore the limit 
	\[	\lim_{j \to \infty} \Phi_{j+1} \circ \Phi_j^{-1}
	\]
	exists uniformly on compacts on 
	\[	\bigcup_{j = 1}^{\infty} \left( \Phi_{j+1} \circ \Phi_1^{-1} \right)^{-1}  (Z_{R_j}) = \camo
	\]
	and 
	\[	\Phi = (\lim_{j \to \infty} \Phi_{j+1} \circ \Phi_j^{-1} ) \circ \Phi_1 = \lim_{j \to \infty} \Phi_{j} 
	\]
	is a holomorphic symplectic automorphism of $\camo$ with $\Phi (\camo^\RR) = \camo^\RR$. 
	Moreover, $(1_j)$ and $(4_j)$ guarantee that $\psi_{j, 1} \circ \Phi_j$ is $\Phi_j$ on $Z_{R_j}^\RR$. Then $(5_j)$ says that $\Phi_j$ approximates $\varphi$ on $Z_{R_j}^\RR$ and thus in the limit $\Phi$ approximates $\varphi$. 
\end{proof}


\appendix 
\section{A different real form}
Let us consider another antiholomorphic involution $\sigma$ on $\mabove$
\[	\sigma ( X, Y, v, w ) = ( \bar{X}, \bar{Y}, \bar{v}, \bar{w} )
\]
and the corresponding subgroup $\gln_r = \gl_n(\RR)$ acting on 
\[	\maboveR = \matnR \oplus \matnR \oplus \RR^n \oplus (\RR^n)^\ast
\]
by the restriction of the $\gln$-action. This gives the moment map
\[	\mu_r \colon \maboveR \to \matnR, \quad (X,Y,v,w) \mapsto [X, Y] + vw
\]
Since the complex quotient $\camo$ is Hausdorff and all $\gln$-orbits in $\caboveB$ are closed, the $\gln$-action is proper. Thus $\gln_r$ also acts properly and freely on $\caboveBr = \mu_r^{-1}(- I_n)$. Notice that we consider here the complex Calogero--Moser space $\camo'$ with a different rank condition $[ X, Y] + vw = - I_n$, which is isomorphic to $\camo$ in the main text. 

Moreover, the real symplectic form on $\maboveR$ is given by $\omega_r = \restr{\omega}{\maboveR}$. Then the  real symplectic reduction $\caboveBr / \gln_r$ is a real symplectic manifold $\camo^r$ equipped with an induced symplectic form also denoted by $\omega_r$. 

The above conjugation maps a $\gln$-orbit to another $\gln$-orbit
\[	\sigma ( g \cdot z ) = \bar{g} \cdot \sigma (z)
\]
for $g \in \gln$ and $z \in \mabove$. Hence it induces a conjugation on $\camo$. It is clear that the conjugation commutes with the $\gln_r$-action. From the above we have $\caboveBr = (\caboveB)^{\sigma}$, the fixed-point set of the conjugation on $\caboveB = \mu^{- 1}(-I_n)$. The map $\jmath' \colon \camo^r \to \camo'$ takes a $\gln_r$-orbit to the $\gln$-orbit containing it. 
\[
\begin{CD}
	\caboveB @>\sigma>>	 \caboveB \\
	@VVV 			 @VVV\\
	\camo'	@>\sigma>>	\camo'
\end{CD}
\quad \quad \quad
\begin{CD}
	\caboveBr @>j'>>	 \caboveB  \\
	@VVV 			 @VVV\\
	\camo^r	@>\jmath'>>	\camo'
\end{CD}
\]

\begin{lemma}
	The map $\jmath'$ is surjective. 
\end{lemma}
\begin{proof}
	Take a $\sigma$-stable $\gln$-orbit $\gln \cdot z$. We show that it meets $\caboveBr$ at one point. As in the proof of Theorem \ref{theorem: realform} using GIT, we see from $\| \sigma(z) \| = \| z \|$ that for $z_0 \in \mu_1^{-1}(0) \cap \caboveB$ there exists $h_0 \in \un$ such that 
	\[	\sigma ( z_0 ) = h_0 \cdot z_0
	\]
	Then 
\[
	z_0 = \sigma (\sigma (z_0)) = \sigma (h_0 \cdot z_0) = \bar{h}_0 \cdot \sigma (z_0) = \bar{h}_0 h_0 \cdot z_0
\]
which implies that $ \bar{h}_0 h_0 = I_n$. This in turn gives $h_0^{-1} = h_0^{\ast} = \bar{h}_0$. 
Then $h_0$ is also symmetric, thus there exists another symmetric unitary matrix $s$ such that $h_0 = s^2$. Replace $z_0$ by $s \cdot z_0$ 
	\[	\sigma ( s \cdot z_0 ) = \bar{s} s^2 \cdot z_0 = s \cdot z_0 
	\]
	Namely $s \cdot z_0 \in \caboveBr$. 
\end{proof}

The rest of the proof for the following is similar to the one for Theorem \ref{theorem: realform}. 
\begin{theorem} \label{theorem: realform2}
    The complex Calogero--Moser space $(\camo', \omega, \sigma)$ is a symplectic complexification of the smooth symplectic manifold $(\camo^r, \omega_r)$. 
\end{theorem}

 The computation in \cite{CaloSymplo} shows that any algebraic $\gln$-invariant on $\cabove'$ is contained in the complex Lie algebra generated by the four functions in $\tr Y, \tr Y^2, \tr X^3, (\tr X)^2$, and all formulae appearing in the calculation leading to it come with real coefficients. This implies the $\tau$-symplectic density property of $\camo'$ with respect to $\camo^r$. On the other hand, the antiholomorphic involution $\sigma$ induces the complex conjugation on the algebra generators $\tr X^j Y^k$. Hence the real form $\camo^r$ is the intersection of the complex Calogero--Moser space $\camo'$ with $\RR^q$ (observe that $\caboveB= \mu^{-1}(- I_n)$ and $\caboveF= \mu^{-1}(i I_n)$ are different sets in $\mabove$), where $q$ is the cardinality of a minimal generating set $\mathcal{G}$ for the $\CC$-algebra $\CC[\cabove']^{\gln}$ with generators of the form $\tr X^j Y^k$. Therefore the same construction of holomorphic Hamiltonian automorphisms to approximate Hamiltonian diffeomorphisms on $\camo^r$ stays intact. 

\begin{theorem} \label{theorem: HamCarleman2}
	Let $\varphi$ be a Hamiltonian diffeomorphism of the symplectic manifold $( \camo^r, \omega_r)$. Then for a positive continuous function $\epsilon$ on $\camo^r$, there exists a holomorphic symplectic automorphism $\Phi \in \taut_\omega (\camo')$ such that 
	\[	\| \Phi - \varphi \|_{C^k(p)} < \epsilon (p)	\]
	for all $p$ in $\camo^r$.  
\end{theorem}

\section*{Acknowledgements.}
I would like to thank Frank Kutzschebauch for valuable discussions and patient guidance, and Rafael Andrist for suggesting the topic. I also extend my thanks to Gene Freudenburg for a conversation on ring generators and real forms. Furthermore, I am grateful to Franc Forstneri{\v c} for pointing out an imprecise notion in an earlier version.

\end{document}